\documentclass[a4paper, 12pt, titlepage, fleqn]{article}
\usepackage{times}
\usepackage{amsthm}
\newtheorem{thm}{Theorem}
\newtheorem{cor}{Corollary}

\newtheorem{lemma}[thm]{Lemma}
\newtheorem{prop}{Proposition}
\newtheorem{remark}{Remark}

\usepackage{amssymb,amsmath}

\DeclareMathOperator{\F}{\mathbb{F}}

\DeclareMathOperator{\Tr}{Tr}

\begin{document}

	\baselineskip=16.3pt
\parskip=14pt
\begin{center}
	\section*{Explicit Maximal and Minimal Curves of Artin-Schreier Type from Quadratic Forms}
	{\large 
		
		Daniele Bartoli\footnote{Dipartimento di Matematica e Informatica, Universit\`a degli Studi di
			Perugia,\\ Via Vanvitelli 1, Perugia, 06123   Italy.  e-mail: daniele.bartoli@unipg.it, Research partially supported by Ministry for Education, University and Research of Italy (MIUR) (Project PRIN 2012 \textit{Geometrie di Galois e strutture di incidenza}-Prot. N.2012XZE22K$\textunderscore$005) and by the Italian National Group for Algebraic and Geometric Structures and their Applications (GNSAGA-INdAM).} , Luciane Quoos \footnote{Instituto de Matem\'atica, Universidade Federal do Rio de Janeiro, 
			Av. Athos da Silveira Ramos 149, Centro de Tecnologia - Bloco C, Ilha do
			Fund\~ao, Rio de Janeiro, RJ 21941-909.  Brazil. e-mail: luciane@im.ufrj.br.} Z\"{u}lf\"{u}kar Sayg{\i}\footnote{Department of Mathematics,
			TOBB University of Economics and Technology, e-mail: zsaygi@etu.edu.tr.}, Emrah Sercan Y\i lmaz \footnote{Department of Mathematics and Statistics, University College Dublin, e-mail:~emrahsercanyilmaz@gmail.com, Research supported by Science Foundation Ireland Grant 13/IA/1914.}
		 }
\end{center}

\subsection*{Abstract}
In this work we present explicit examples of maximal and minimal curves over finite fields in odd characteristic. The curves are of Artin-Schreier type and the construction is closely related to quadratic forms from $\mathbb{F}_{q^n}$ to $\mathbb{F}_q$.

\section{Introduction}\label{sec.intro}

In the interaction between algebraic curves over finite fields and applications in coding theory, cryptography, quasi-random numbers and related areas it is important to know the number of rational points of the curve (see, for example, \cite{Cem,NX1,NX2,S,TV}). Artin-Schreier curves over finite fields is a central theme and many of the known constructions of maximal or minimal curves are closely related to quadratic forms. 
Recently, some characterizations and classification results were obtained in the literature.
Let $\F_{q}$ denote the finite field with $q$ elements. For $q=2^t$ a full classification of quadratic forms from $\mathbb{F}_{q^k}$ to $\mathbb{F}_{q}$ of codimension 2 is provided in the following cases: all the coefficients are from $\F_2$ or at least three are in $\F_4$; as an application maximal and minimal curves are obtained, see \cite{F1,F2,OS1,OS2,OS3}. Latter on some results on quadratic functions and maximal Artin-Schreier curves over finite fields having odd characteristic are presented in \cite{AM2015} and \cite{W}. In \cite{OS4} by using some techniques developed in \cite{COJPAA} a Conjecture presented in \cite{W} is proved and explicit classes of maximal and minimal Artin-Schreier type curves over finite fields having odd characteristics are presented. 

Throughout this paper by a curve we mean a smooth geometrically irreducible and projective curve over a finite field of odd characteristic. 
For a positive integers $m$ consider the $\F_q$-linearized polynomial of degree $q^m$
\begin{equation*}
S(x)=s_0 x + s_1 x^{q} + \cdots + s_m x^{q^m} \in \F_{q^n}[x].
\end{equation*}

In this work we consider the Artin-Schreier type curves $\mathcal{X}$ defined as
\begin{equation}
 \label{chi}
\mathcal{X}:\quad y^q - y= x S(x)=\sum_{i=0}^{h} s_i x^{q^i+1}.
\end{equation}
First note that such curves have a unique singular point at infinity (which is $\mathbb{F}_{q^n}$-rational). Also, there is a unique place centered on it; see for instance \cite[Proposition 3.7.10]{S}. This means that the number of $\mathbb{F}_{q^n}$-rational points of $\mathcal{X}$ equals the number of degree one places in the corresponding function field.
These curves are related with the quadratic forms (see, Section \ref{sec.prelim})
\begin{equation} \label{Q}
Q(x)=\Tr (x S(x))
\end{equation}
where $\Tr(\cdot)$ denotes the trace map from $\F_{q^n}$ to $\F_q$, that is, 
$\Tr(x)=x+x^q+\cdots+x^{q^{n-1}}$.

Let $N(\mathcal{X})$ be the number of $\F_{q^n}$-rational points of the curve $\mathcal{X}$ and $N(Q)$ denote the cardinality
\begin{equation*}
N(Q) = \left|\left\{ x \in \F_{q^n} \mid \Tr \left(xS(x)\right) = 0 \right\}\right|.
\end{equation*}
From Hilbert's Theorem 90 we obtain
\begin{equation*}
N(\mathcal{X})=1+q N(Q),
\end{equation*}
and furthermore by the Hasse-Weil inequality we know that 
\begin{equation*}
q^n+1-2g(\mathcal{X})\sqrt{q^n} \leq N(\mathcal{X}) \leq q^n+1+2g(\mathcal{X})\sqrt{q^n}
\end{equation*}
where $g(\mathcal{X})$ is the genus of $\mathcal{X}$.

Curves attaining the Hasse-Weil bounds have special attention. 
If the number of $\F_{q^n}$ rational points of a curve is $q^n+1+2g(\mathcal{X})\sqrt{q^n}$ then it is called a maximal curve, and if the number of $\F_{q^n}$ rational points of a curve is $q^n+1-2g(\mathcal{X})\sqrt{q^n}$ then it is called a minimal curve. 

In this work we determine examples of minimal and maximal curves of type \eqref{chi}. Our investigation is based on the type of the quadratic form associated with the curve. In particular we generalize curves constructed in \cite{OS4}.

\section{Preliminaries}\label{sec.prelim}

In this section we first present some definitions and facts that we use in this paper connecting Artin-Schreier type curves and quadratic forms. A {\it quadratic form}  $Q: \mathbb{F}_{q^n} \rightarrow \mathbb{F}_q$ is a map such that
\begin{enumerate}
	\item [i)] $Q(ax)=a^2Q(x)$ for all $a \in \mathbb{F}_q$ and $x \in \mathbb{F}_{q^n}$.
	\item [ii)]$B(x,y)=Q(x+y)-Q(x)-Q(y)$ is a bilinear map over $\mathbb{F}_{q^n}$.
\end{enumerate}

The radical $W$ associated to the quadratic form $Q$ is defined as 
\begin{equation*}
W=\left\{ x \in \F_{q^n}: B(x,y)=0 \mbox{   for all $y \in \F_{q^n}$}\right\}.
\end{equation*}
Note that $W$ is an $\F_q$-linear subspace of $\F_{q^n}$ and let $w$ be the $\F_q$-dimension of $W$. The difference $n-w$ is called the codimension of the radical.

For the algebraic curve 
\begin{equation}
 \label{artin}
\mathcal{X}:\quad y^q - y= x S(x)=\sum_{i=0}^{m} s_i x^{q^i+1}.
\end{equation}
we consider the quadratic form $Q: \mathbb{F}_{q^n} \rightarrow \mathbb{F}_q$ given by $Q(x)=Tr(xS(x))$, where $Tr$ denotes the Trace function from $\mathbb F_{q^n}$ to $\mathbb{F}_q$.
In 2007   ~\c{C}ak\c{c}ak and \"Ozbudak, using the classification of quadratic forms, determined the exact value of $N(\mathcal{X})$, the number of $\F_{q^n}$ rational points of the curve $\mathcal{X}$ (see \cite[Theorem 3.1]{COJPAA}). And we obtain 
\begin{equation*}
N(\mathcal{X})=
\begin{cases}
1+q^n\pm (q-1)q^{\frac{n+w}{2}} & $, if $ w $ is even$,\\
1+q^n & $, if $ w $ is odd.$
\end{cases}
\end{equation*}

The curve $\mathcal{X}$ defined on (\ref{artin})
has genus $g(\mathcal{X})=\frac{q-1}{2}q^m$, see  \cite[Proposition 3.7.10]{S} and for even $w$ we obtain: the curve $\mathcal{X}$ is maximal or minimal over $\mathbb{F}_{q^n}$ if and only if the dimension of the $\mathbb{F}_q$-vector space $W$ is $w=2m$. 

Now we present a result about the vector space $W$. Since the proof is short we include it here for the reader's convenience.
\begin{lemma}\cite[Lemma 2.1]{COJPAA} \label{W}
	Let $S(x)=s_0+s_1x^q+ \dots + s_mx^{q^m} \in \mathbb{F}_{q^n}[x]$ and $Q(x)=Tr(xS(x))$ be the quadratic form associated to $S(x)$. The elements in $W=\{ x \in \mathbb{F}_{q^n} \ : \ B(x,y)=0 \ \forall\ y \in \mathbb{F}_{q^n}\}$ are the roots in $\mathbb{F}_{q^n}$ of the polynomial
\begin{equation*}
	\sum_{i=0}^{m-1} s_{m-i}^{q^i}x^{q^{m-i}}+2s_0^{q^m}x^{q^m}+\sum_{i=1}^m s_i^{q^m}x^{q^{m+i}} \in \mathbb{F}_{q^n}[x],
\end{equation*}
	and $W$ has dimension less than $2m+1$.
\end{lemma} 

\begin{proof}
 Write $B(x,y)=Tr(xS(y))+Tr(yS(x))$. 
From $Tr(a^{q^k})=Tr(a), \forall a \in  \mathbb{F}_{q^n}$ and $k=0,\dots d$ and $Tr$ being an additive function, it follows that for any $a, b \in  \mathbb{F}_{q^n}$
\begin{align*}
B(a,b)&=Tr\left(a\sum_{i=0}^m s_ib^{q^i}\right)+Tr\left(b\sum_{i=0}^m s_ia^{q^i}\right)\\&=Tr\left(b\sum_{i=0}^m (s_ia)^{q^{-i}}\right)+Tr\left(b\sum_{i=0}^m s_ia^{q^i}\right)\\
&=Tr\left(b\left(\sum_{i=0}^m (s_ia)^{q^{-i}}+\sum_{i=0}^m s_ia^{q^i}\right)\right).
\end{align*}
For any $a \in  \mathbb{F}_{q^n}$, we have that $B(a,b)=0 \, \forall \, b \in  \mathbb{F}_{q^n}$ if and only if $a$ is a root in $ \mathbb{F}_{q^n} $ of the degree $q^{2d}$ polynomial 
$\sum_{i=0}^m (s_ix)^{q^{-i}}+\sum_{i=0}^m s_ix^{q^i}$, or equivalently, a root of $\sum_{i=0}^m (s_ix)^{q^{d-i}}+\sum_{i=0}^m s_i^{q^m}x^{q^{m+i}}$.
\end{proof}

The following result was proved in \cite{Cem} using some tools from algebraic geometry and was also proved in \cite{OS4} using only elementary tools. 
\begin{prop}\label{prop1}
Let $q$ be a prime power and let $m \geq 1$ be an integer.
Consider the curve $\mathcal{X}$ over $\F_{q^{2m}}$ defined by
\begin{equation*}
\mathcal{X}:\quad y^q - y= x \left( s_0 x + s_1 x^q + \cdots + s_m x^{q^m} \right).
\end{equation*}
Assume that $s_m\ne 0$ and $\mathcal{X}$ is maximal over $\F_{q^{2m}}$. Then $s_0=s_1=\cdots=s_{m-1}=0$ and $s_m+{s_m}^{q^m}=0$. The converse holds as well. 
\end{prop}

\begin{thm}[\cite{OS4}] \label{minimal-thm}
	Let $q$ be a power of an odd prime and $k$, $m$ be positive integers with $m \ge 2k$. Let \begin{equation*}
	S(x)=s_kx^{q^k}+s_{k+1}x^{q^{k+1}}+\cdots+s_{m-k}x^{q^{m-k}} \in\mathbb F_{q^{2m}}[x] \;\;\; \text{ with } s_ks_{m-k}\ne 0.
	\end{equation*} Assume that the radical of the quadratic form $\Tr(xS(x))$ has dimension $2m-2k$ over $\F_q$. Then the curve \begin{equation*}
	\mathcal{X}:y^q-y=xS(x)
	\end{equation*} is a minimal curve over $\mathbb F_{q^{2m}}$.
\end{thm}

\section{Explicit curves from quadratic forms whose radicals have codimension two}\label{sec.main}

Our first result characterizes maximal curves 
from quadratic forms whose radicals have codimension two,
over $\F_{q^{2m}}$.

\begin{thm}\label{thm1}
Let $q$ be a power of an odd prime, and let $m\geq 2$ be a positive integer.
Let \begin{equation*}
S(x)=s_0 x + s_1 x^q+\cdots+s_{m-1} x^{q^{m-1}}\in\F_{q^{2m}}[x]\quad \mbox{with}\quad s_0 s_{m-1}\ne 0.
\end{equation*}
Then the curve 
\begin{equation*}
\mathcal{X}:\quad y^q - y= x S(x)
\end{equation*}
is a maximal curve over $\F_{q^{2m}}$ if and only if the following equations are satisfied
\begin{equation}\label{eq1.theorem1}
\begin{array}{ll}
&c^q s_1= -\left(c^{2q}s_0^q+s_0\right) \\
&c^{q^2} s_2 = -\left(2c^{q}s_0^q+c^{q^2+q}s_1^q + s_1\right) \\
&c^{q^3} s_3 = -\left(c^{q}s_1^q+c^{q^3+q}s_2^q + s_2\right) \\
&\vdots \\
&c^{q^i} s_i = -\left(c^{q}s_{i-2}^q+c^{q^{i}+q}s_{i-1}^q + s_{i-1}\right) \\
&\vdots \\
&c^{q^{m-1}} s_{m-1} = -\left(c^{q}s_{m-3}^q+c^{q^{m-1}+q}s_{m-2}^q + s_{m-2}\right) \\
\end{array} 
\end{equation}
and 
\begin{equation}\label{eq2.theorem1}
\begin{array}{ll}
&c^q s_{m-2}^q + c^{q^{m}+q}s_{m-1}^q + s_{m-1}=0 \\
&c s_{m-1} + \left(c s_{m-1}\right)^{q^m}=0
\end{array} 
\end{equation}
for some $c\in \F_{q^{2m}}\setminus\{0\}$.
\end{thm}
\begin{proof}
Let $\mathbb E_1=\F_{q^{2m}}(x,y)$ with $y^q - y= xS(x)$ be the function field of $\mathcal{X}$. As the dimension of the radical is $2m-2$, $\deg(S(x))=q^{m-1}$ and
$s_{m-1}\ne 0$,  $\mathbb E_1$ (or equivalently $\mathcal{X}$) is either maximal or minimal over $\F_{q^{2m}}$. Using \cite[Proposition 5.1]{COJPAA} we can construct an extension field $\mathbb E_2$ of $\mathbb E_1$ such that 
\begin{equation*}
\mathbb E_2 \mbox{ is maximal (minimal)} \Leftrightarrow \mathbb E_1 \mbox{ is maximal (minimal)}. 
\end{equation*}
Moreover an affine equation for $\mathbb E_2$ is also given: $\mathbb E_2=\F_{q^{2m}}(z,t)$ with
\begin{equation*}
t^q-t=zR(z).
\end{equation*}
Here \cite[Proposition 5.1]{COJPAA} proves existence of $c\in \F_{q^{2m}}\setminus\{0\}$ such that 
\begin{equation}\label{eq3.theorem1}
\begin{array}{rl}
  D(x)^q &= S(x^q + cx)-c s_0x\quad \mbox{and} \\
  R(x) &= cS(x^q + cx) + D(x) + c s_0 x^q 
 \end{array} 
\end{equation}
in the polynomial ring $\F_{q^{2m}}[x]$. Then using (\ref{eq3.theorem1}) we obtain that 
\begin{align}
  \hspace{-7mm}D(x)^q &= \sum_{i=0}^{m-1} s_i x^{q^{i+1}} + \sum_{i=0}^{m-1} s_i (cx)^{q^{i}} - c s_0x \quad \mbox{and} \\
  \hspace{-7mm}R(x) &=  c \left(\sum_{i=0}^{m-1} s_i x^{q^{i+1}} + \sum_{i=0}^{m-1} s_i (cx)^{q^{i}}\right) 
			+ \left(\sum_{i=0}^{m-1} s_i^{(1/q)} x^{q^{i}} + \sum_{i=1}^{m-1} s_i^{(1/q)} (cx)^{q^{i-1}}\right) + c s_0 x^q.
\end{align}
Using Proposition \ref{prop1}, $\mathbb E_2$ is maximal if and only if the coefficients of $R(x)$ satisfies the equations in 
(\ref{eq1.theorem1}) and (\ref{eq2.theorem1}), which completes the proof.
\end{proof}

If we take all the coefficients $s_i$ of $S(x)$ in $\F_{q^m}$ we obtain the following explicit classifications  in Corollaries \ref{cor1}, \ref{cor2} and \ref{cor3}. These results include the maximal curves obtained in \cite{W} as a very special subcase. Also note that in \cite{W} only the case $q=p$ (prime case) is considered under the condition that $\gcd(p,n)=\gcd(p,2m)=1$. Here we have no such condition.

\begin{cor}\label{cor1}
Let $q$ be a power of an odd prime and  let $m\geq 2$ be a positive integer.
Let
\begin{equation*}
S(x)=s_0 x + s_1 x^q+\cdots+s_{m-1} x^{q^{m-1}}\in\F_{q^m}[x]\quad \mbox{with}\quad s_0 s_{m-1}\ne 0.
\end{equation*}
Then the radical of the quadratic form $\Tr(xS(x))$ has dimension $2m-2$ over $\F_q$ and the curve 
\begin{equation*}
\mathcal{X}:\quad y^q - y= x S(x)
\end{equation*}
is a maximal curve over $\F_{q^{2m}}$ if and only if $q \equiv 3 \mod 4$, $m$ is odd, $s_0\in\F_{q^m}\setminus\{0\}$ and for $1\le i\le m-1$ we have
\begin{equation}
s_i=\left\{
\begin{array}{cc}
0 & \mbox{if $i$ is odd},\\
2s_0^{(q^i+1)/2} & \mbox{if $i$ is even}.\\
\end{array} \right.
\end{equation}
\end{cor}
\begin{proof}

Let $m \geq 2$. Since $s_{m-1} \in F_{q^m}^\ast$, we have  \begin{equation*}
cs_{m-1}+(cs_{m-1})^{q^m}=(c+c^{q^m})s_{m-1}=0
\end{equation*} an so  
$c+c^{q^m}=0$.
Moreover, we have \begin{equation*}
c^qs_1+c^{2q}s_0^q+s_0=0. 
\end{equation*}
If we take the powers $q^{m-1}$ and $q^{2m-1}$ respectively, since $s_0,s_1 \in \mathbb F_{q^m}$ we will obtain the equations \begin{equation*}
-cs_1^{q^{m-1}}+c^2s_0+s_0^{q^{m-1}}=0
\end{equation*}
and \begin{equation*}
cs_1^{q^{m-1}}+c^2s_0+s_0^{q^{m-1}}=0.
\end{equation*}
These equations gives us \begin{equation*}
s_1=0 \;\;\; \text{ and } \;\;\; c^2=-s_0^{q^{m-1}-1}. 
\end{equation*}
This shows that the case $m=2$ cannot happen since $s_0s_1=0$. Let us assume $m\ge 3$. For $i=2,\dots,m-1$ we have the equations \begin{equation*}
c^{q^i}s_i+c^qs_{i-2}^q+c^{q^i+q}s_{i-1}^q+s_{i-1}=0.
\end{equation*}
These equations give us when $q \equiv 3 \mod 4$ 
\begin{equation*}
s_i=\begin{cases} 0 &\text{if $i$ is odd}, \\ 2s_0^{(q^i+1)/2} &\text{if $i$ is even,}   \end{cases} 
\end{equation*} and  when $q \equiv 1 \mod 4$ \begin{equation*}
s_i=\begin{cases} 0 &\text{if $i$ is odd}, \\ (-1)^{i/2}\; 2s_0^{(q^i+1)/2} &\text{if $i$ is even}   \end{cases}
\end{equation*}  where $i\in\{1,\dots,m-1 \}.$ 
	
Since $s_{m-1} \ne 0$, we have $m-1$ is even, so $m$ must be odd. Moreover, since $s_{m-2}=0$ and $c^{q^m}=-c$, the equation \begin{equation*}
c^qs_{m-2}^q+c^{q^m+q}s_{m-1}^q+s_{m-1}=0
\end{equation*} gives us \begin{equation*}
c^{q+1}=s_{m-1}^{1-q}
\end{equation*} and so \begin{equation*}
(-s_0^{q^{m-1}-1})^{(q+1)/2}=(2s_0^{(q^{m-1}+1)/2})^{1-q}
\end{equation*} and so \begin{equation*}
(-1)^{(q+1)/2}=1.
\end{equation*} This can only happen when $q\equiv 3 \mod 4$.\\
	
Assume  $q \equiv 3 \mod 4$, $m$ is odd and \begin{equation*}
s_i=\begin{cases} 0 &\text{ if $i$ is odd,} \\ 2s_0^{(q^i+1)/2} &\text{ if $i$ is even}   \end{cases}
\end{equation*} and let \begin{equation*}
Q(x)=\Tr\left(x\sum_{i=0}^{m-1}s_ix^{q^i}\right).
\end{equation*}  
Then \begin{align*}
&\Tr(Q(x+y)-Q(x)-Q(y))= \Tr\left(2s_0xy+\sum_{i=1}^{(m-1)/2}s_0^{(q^{2i}+1)/2}(x^{q^{2i}}y+xy^{q^{2i}})\right) \\&=\Tr\left( s_0^{(q^m-1)/2}y^{q^m} \sum_{i=0}^{m-1}s_0^{(q^{2i+1}+1)/2}x^{q^{2i+1}}\right)=\Tr\left(s_0^{(q^m-1)/2} sy^{q^m} \sum_{i=0}^{m-1}(sx)^{q^{2i+1}}\right)	
\end{align*}
	with fixing a square root $s$ of $s_0$ in $\mathbb F_{q^{2m}}$. Since \begin{equation*}
	(sx)^{q^{2m}}-(sx)=s(x^{q^{2m}}-x)
	\end{equation*} and since \begin{equation*}
	\deg \left((x+x^3+\cdots+x^{2m-1},x^{2m}-1)\right)=2m-2,
	\end{equation*} we have the result.
\end{proof}

\textit{Remark.} The maximal curves in Corollary \ref{cor1} have genus $q^{m-1}(q-1)/2$ . By \cite[Theorem 6.12]{COJPAA} such curves are covered by the corresponding Hermitian curve. Note that subcovers of the Hermitian curves with the same genus could be also obtained using \cite[Proposition 3.1]{GSX}.

\begin{cor}\label{cor2}
Let $q\equiv 1 \mod 4$ be a power of an odd prime and 
 let $m\geq 2$ be a positive odd integer.
Let
\begin{equation*}
S(x)=s_0 x + s_1 x^q+\cdots+s_{m-1} x^{q^{m-1}}\in\F_{q^m}[x]\quad \mbox{with}\quad s_0 s_{m-1}\ne 0
\end{equation*}
where \begin{equation*}
s_i=\begin{cases} 0 &\text{ if $i$ is odd}, \\ 2s_0^{(q^i+1)/2} &\text{ if $i$ is even} \end{cases}
\end{equation*} for $i=1,\cdots,m-1$. Then the radical of the quadratic form $\Tr(xS(x))$ has dimension $2m-2$ over $\F_q$ and the curve  is a minimal curve over $\F_{q^{2m}}$.

\end{cor}

\begin{proof}
The calculation of the dimension of $\Tr(xS(x))$ in Corollary \ref{cor1} works here too. Since its dimension is $2m-2$ and since the curve $\mathcal{X}$ is not maximal by Corollary \ref{cor1}, $\mathcal{X}$ is minimal.
\end{proof}

\begin{cor}\label{cor3}
	Let $q\equiv 1 \mod 4$ be a power of an odd prime and
	 let $m\geq 2$ be a positive even integer.
	Let
\begin{equation*}
	S(x)=s_0 x + s_1 x^q+\cdots+s_{m-1} x^{q^{m-1}}\in\F_{q^m}[x]\quad \mbox{with}\quad s_0 s_{m-1}\ne 0
\end{equation*}
	where \begin{equation*}
	s_i=\begin{cases} 0 &\text{ if $i$ is even}, \\ s_1^{(q^i+1)/(q+1)} &\text{ if $i$ is odd}   \end{cases}
	\end{equation*} for $i=0,\dots,m-1$. Then the radical of the quadratic form $\Tr(xS(x))$ has dimension $2m-2$ over $\F_q$ and the curve  is a minimal curve over $\F_{q^{2m}}$.
	
\end{cor}

\begin{proof}
 Let \begin{equation*}
 Q(x)=\Tr\left(x\sum_{i=0}^{m-1}s_ix^{q^i}\right).
 \end{equation*} Then \begin{align*}
 &\hspace{-5mm}\Tr(Q(x+y)-Q(x)-Q(y)= \Tr\left(\sum_{i=1}^{m/2}s_1^{(q^{2i-1}+1)/(q+1)}(x^{q^{2i-1}}y+xy^{q^{2i-1}})\right)\\&\hspace{-5mm}=\Tr\left( s_1^{(q^m-1)/(q+1)}y^{q^m} \sum_{i=0}^{m-1}s_1^{(q^{2i+1}+1)/(q+1)}x^{q^{2i+1}}\right)=\Tr\left(s_1^{(q^m-1)/(q+1)} sy^{q^m} \sum_{i=0}^{m-1}(sx)^{q^{2i+1}}\right)
 \end{align*}
with fixing a $(q+1)$-th root of $s_1$ in $\mathbb F_{q^{2m}}$, we called it $s$. Since \begin{equation*}
(sx)^{q^{2m}}-(sx)=s(x^{q^{2m}}-x)
\end{equation*} and since \begin{equation*}
\deg \left((x+x^3+\cdots+x^{2m-1},x^{2m}-1)\right)=2m-2,
\end{equation*} we have the result.
\end{proof}

\begin{remark}
Corollary \ref{cor1} and Corollary \ref{cor2} are true when $m=1$.
The proof can be found in \cite{MY} (Lemma 5)
where $p$ can be replaced by $q$ and $x^2$ can be replaced by $s_0 x^2$
for $s_0\in \F_q$.
\end{remark}

\section{Explicit curves using cyclotomic polynomials}

Assume that $d$ is not divisible by the characteristic of $\F_q$.
The $d$-th cyclotomic polynomial $\Phi_d(x)$ over $\F_q$ is defined as 
\begin{equation*}
\Phi_d(x)=\prod\limits_{\substack{s=1\\ \gcd(s,d)=1}}^{d}(x-\xi^s),
\end{equation*}
where $\xi$ is a primitive $d$th root of unity over $\F_q$. In particular $\Phi_d(x)$ is always a divisor of $x^d -1$, but not necessarily irreducible over $\F_q$.
The following are well-known results about cyclotomic polynomials (see, for example \cite{LN}).

\begin{lemma}
	The coefficients of the cyclotomic polynomial $\Phi_d(x)$ are 
	in $\F_p$ for all $d \ge 1$ with $\gcd (d,p)=1$.
\end{lemma}


\begin{lemma} \label{symmetry}
	Let $d>1$ 
be relatively prime to $p$,	
and set $\Phi_d(x)=\sum\limits_{k=0}^{\phi(d)}a_kx^k$. Then $a_{\phi(d)-i}=a_i$ for all $0 \le i \le \phi(d)$.
\end{lemma}

If $\Phi_d(x)=\sum\limits_{k=0}^{\phi(d)}a_kx^k$ then we define 
$\varphi_d(x)=\sum\limits_{k=0}^{\phi(d)}a_kx^{q^k}$.

\begin{thm}\label{cyclotomic.main.thm}
	Let $k$ be a positive even integer and $d$ be a positive divisor of $k$ which is bigger than $2$. Then the curve \begin{equation*}
	\mathcal{X}:y^q-y=x\sum\limits_{j=0}^{\frac{n}{2k}-1} \varphi_d(x)^{q^{a+kj}}
	\end{equation*} is minimal over $\mathbb F_{q^n}$ where $n$ divisible by $2k$ and $\phi(d)+2a=k$.
\end{thm}
\begin{proof}
By Lemma \ref{W} we have\begin{equation*}
 W	=\left\{ x \in \mathbb F_{q^n} \; | \; \sum_{j=0}^{\frac{n}{k}-1}(\varphi(x))^{q^{kj}}=0 \right\}.
\end{equation*} Therefore the corresponding associated polynomial to $\sum_{j=0}^{\frac{n}{k}-1}(\varphi(x))^{q^{kj}}$
	is 
	\begin{equation*}
	\Phi_d(x)(1+x^k+\cdots+x^{n-k})
	\end{equation*}	and $\deg\gcd \left(\Phi_d(x)(1+x^k+\cdots+x^{n-k}),x^n-1\right)=n-k+\phi(d)=n-2a.$ 
	Now the result follows from Theorem \ref{minimal-thm}.  
\end{proof}

\begin{remark}\label{remark1}
Here we remark that Theorem \ref{cyclotomic.main.thm} includes the explicit classes of minimal curves given in \cite[Theorem 3.4 and Theorem 3.5]{OS4}. If we use $\phi_2(x)=x^2 - x + 1$, that is, $\varphi_2(x)=x^{q^2} - x^{q} + x$, then Theorem \ref{cyclotomic.main.thm} reduces to \cite[Theorem 3.4]{OS4}. Furthermore, if we use $\phi_4(x)=x^2 + x + 1$, that is, $\varphi_4(x)=x^{q^2} + x^{q} + x$, then Theorem \ref{cyclotomic.main.thm} reduces to \cite[Theorem 3.5]{OS4}.
\end{remark}

\begin{thm}\label{cyclotomic2.thm}
	Let $k$ be a positive even integer and $d\geq 2$ a divisor of $k$. Then the curve \begin{equation*}
	\mathcal{X}:y^q-y=x\sum\limits_{j=0}^{\frac{n-k}{2k}-1} \varphi_d(x)^{q^{a+kj}}+x\sum_{i=0}^{\frac{\phi(d)}{2}-1} c_ix^{q^{k-a-i}}+ \frac{c_{\phi(d)/2}}{2}x^2
	\end{equation*} is minimal over $\mathbb F_{q^{2n}}$ where $n\equiv k \mod 2k$, $n>k$ and $\phi(d)+2a=2k$.
\end{thm}
\begin{proof}
By Lemma \ref{W} we have \begin{equation*}
 W=\left\{ x \in \mathbb F_{q^n} \; | \; \sum_{j=0}^{\frac{n}{k}-1}(\varphi(x))^{q^{kj}}=0 \right\}.
\end{equation*} Therefore the corresponding associated polynomial to $\sum_{j=0}^{\frac{n}{k}-1}(\varphi(x))^{q^{kj}}$
	is 
\begin{equation*}
	\Phi_d(x)(1+x^k+\cdots+x^{n-k})
\end{equation*}
	and $\deg\gcd \left(\Phi_d(x)(1+x^k+\cdots+x^{n-k}),x^n-1\right)=n-k+\phi(d)=n-k+\phi(d).$   Since $W \subset \mathbb F_{q^n}$ and the dimension of $W$ over $\mathbb F_q$ is even, $\mathcal{X}$ is maximal or minimal over $\F_{q^n}$ and hence it is minimal over $\mathbb F_{q^{2n}}$.
\end{proof}
\begin{remark}
	Here we remark that Theorem \ref{cyclotomic2.thm} includes the explicit classes of minimal curves given in \cite[Theorem 3.7 and Theorem 3.8]{OS4}. Similar to Remark \ref{remark1} if we use $\phi_2(x)=x^2 - x + 1$ and $\phi_4(x)=x^2 + x + 1$, then Theorem \ref{cyclotomic2.thm} reduces to the minimal curves given in \cite[Theorem 3.7]{OS4} and \cite[Theorem 3.8]{OS4} respectively.
\end{remark}

\section{Some generalizations}

In the previous section, the proofs work for divisors of $x^k-1$ that are
symmetric in the coefficients but are not necessarily cyclotomic polynomials.
Therefore, in the following theorems we start from divisors of $x^k-1$, where $k\geq 2$ divides $n$. We consider an integer $r\geq 1$ and 
\begin{equation}\label{Requests_f}
f(x)=\sum_{i=0}^{2r}a_i x^i\in \mathbb{F}_q[x], \quad \text{where } f(x) \mid x^{k}-1, \text{ and }  a_{r-i}=a_{r+i} \ \forall \ i=1,\ldots, r.
\end{equation}

\begin{thm}\label{Th:General2}
Let $n \geq 2$ be even and $k \equiv 2 \pmod 4$  a divisor of $n/2$. Let 
\begin{equation*}
G(x)=\sum_{i=1}^{r}a_{r+i}x^{q^{k/2-i}}+ a_r x^{q^{k/2}}+\sum_{i=1}^{r}a_{r+i}x^{q^{k/2+i}}.
\end{equation*}
Then the curve $\mathcal{X}_{f,k}$  defined by the affine equation $y^q-y=x\sum_{j=0}^{\frac{n}{2k}-1} G(x)^{q^{jk}}$  is  minimal over $\mathbb F_{q^n}$.
\end{thm}
\begin{proof}
The genus of the curve $\mathcal{X}_{f,k}$ is 
$g=\frac{q-1}{2}q^{\frac{n}{2}-\frac{k}{2}+r}$.
For $w$ the $\mathbb{F}_q$-dimension of the radical $W$ associated to the quadratic form $Q(x)=Tr(xS(x))$ we have: $\mathcal{X}$ is  minimal or maximal over $\mathbb{F}_{q^n}$ if and only if $w=n -k+2r$.
We have 
\begin{align*}
W&=\left\{ x \in \mathbb{F}_{q^n} : \sum_{j=0}^{\frac{n}{k}-1} G(x)^{q^{jr}}=0  \right\}\\&=
\left\{ x \in \mathbb{F}_{q^n} : \sum_{j=0}^{\frac{n}{k}-1} \left(a_rx^{q^r}+ \sum_{i=1}^r (a_{r-i}x^{q^{r-i}}+ a_{r+i}x^{q^{r+i}}) \right)^{q^{jk}}=0  \right\}.
\end{align*}
Therefore the corresponding associated polynomial to $\sum_{j=0}^{\frac{n}{k}-1} \left(a_rx^{q^r}+ \sum_{i=1}^r (x^{q^{r-i}}+ x^{q^{r+i}}) \right)^{q^{jk}}$
is \begin{equation*}
f(x)(1+x^{k}+x^{2k}+\cdots+x^{n-k})
\end{equation*}
and  $\deg\left((\gcd\left(f(x)(1+x^{k}+x^{2k}+\cdots+x^{n-k}),x^{n}-1\right)\right)=n-k+2r=w$. This shows that the curve $\mathbb{X}_{f,k}$ is either maximal of minimal over $\mathbb{F}_{q^n}$. Since the highest and the lowest powers in $S(x)$ are $q^{\frac{n}{2}-\frac{k}{2}+k}$ and $q^{\frac{k}{2}-k}$, by Theorem \ref{minimal-thm} we conclude that $\mathcal{X}_{f,k}$ is minimal.
\end{proof}

Now we construct a family of curves over $\mathbb F_{q^n}$ that are either maximal or minimal over $\mathbb F_{q^n}$ and. We omit the proof since it is very similar to the proof of Theorem \ref{Th:General2}.
\begin{thm}\label{Th:General}
Let $4\leq n=(2s+1)k$ be even. Let 
\begin{eqnarray*}
G(x)&=&\frac{a_r}2 x+\sum_{i=1}^{r} a_{r+i} x^{q^{i}},\\
\widetilde G(x)&=&\sum_{i=1}^{r} a_{r-i} x^{q^{k-i}} +a_r x^{q^k}+\sum_{i=1}^{r} a_{r+i} x^{q^{k+i}}.\\
\end{eqnarray*}
Then the curve $\mathcal{Y}_{f,k}$  of affine equation $y^q-y=x\left(\sum_{j=0}^{\frac{n-k}{2k}-1} \left(\widetilde G(x)\right)^{q^{k j}}\right)+xG(x)$  is either maximal or minimal.
\end{thm}

Finally, we give some examples of polynomials $f(x)$ satisfying the properties in \eqref{Requests_f}.
\begin{prop}
Let $r, s, k \geq 1$ be integers. The following polynomials $f(x)$ satisfy \eqref{Requests_f} in the following cases. 
\begin{enumerate}
\item [i)] $f(x)=\sum_{i=0}^{2r}x^i$ where $2r+1 \mid k$.
\item [ii)] $f(x)=\sum_{i=0}^{2r}(-1)^ix^i$ where $2(2r+1) \mid k.$
\item [iii)] $f(x)=\sum_{i=0}^{2r/s}x^{is}$ where $s\mid r, \, s(2r+1) \mid k.$
\item [iv)] $f(x)=\sum_{i=0}^{2r/s}(-1)^ix^{is}$ where $s\mid r$ and $ 2s(2r+1) \mid k.$
\item [v)] $f(x)=x^{2r}+\left(\sum_{i=2}^{2r-2}x^{i}\right)+1$  where $s=\left\{\begin{array}{ll} 6, & r\equiv0,1 \pmod{3}\\ 2, &r\equiv 2\pmod{3}\end{array}\right.$ and $  s(2r-1)  \mid k.$
\item [vi)] $f(x)=x^{2r}+\left(\sum_{i=2}^{2r-2}(-1)^ix^{i}\right)+1$ where $s=\left\{\begin{array}{ll} 6, & r\equiv0,1 \pmod{3}\\ 2, &r\equiv 2\pmod{3}\end{array}\right. $ and $ s(2r-1)  \mid k.$ 
\end{enumerate}
\end{prop}
\begin{proof}
The first four statements follows immediately from the factorization of $x^k-1$. The last two itens are proved as folllows.
\begin{enumerate}
\item [v)] We have that 
$f(x)=x^{2r}+\left(\sum_{i=2}^{2r-2}x^{i}\right)+1=(1+x+x^2+\cdots +x^{2r-2})(x^2-x+1)$. 
Suppose  $r\equiv 0,1 \pmod 3$, then $6(2r-1)\mid k$. Since $x^{6(2r-1)}-1$ divides $x^{k}-1$ it is enough to show that $f(x) \mid x^{6(2r-1)}-1$. We have that 
\begin{eqnarray*}
x^{6(2r-1)}-1&=&(x^{2r-1}-1)(1+x^{2r-1}+x^{2(2r-1)}+x^{3(2r-1)}+x^{4(2r-1)}+x^{5(2r-1)})\\
&=&(x-1)(1+x+x^2+\cdots +x^{2r-2})(1+x^{2r-1}+x^{2(2r-1)})(1+x^{3(2r-1)})\\
&=&(x-1)(1+x+x^2+\cdots +x^{2r-2})(1+x^{2r-1}+x^{2(2r-1)})\cdot\\
&& \qquad\qquad \qquad \cdot(1+x^3)(1-x^3+x^6-\cdots+x^{3(2r-1)-3})
\end{eqnarray*}
and $f(x)\mid x^{6(2r-1)}-1$. 
Suppose now $r\equiv 2 \pmod 3$ then $3\mid 2r-1$, $2(2r-1) \mid k$. Since $x^{2(2r-1)}-1$ divides $x^{k}-1$ it is enough to show that $f(x) \mid x^{2(2r-1)}-1$. We have that
\begin{eqnarray*}
x^{2(2r-1)}-1&=&(x^{2r-1}-1)(x^{2r-1}+1)\\
&=&(x-1)(1+x+\cdots +x^{2r-2})(1+x^3)(1-x^3+\cdots+x^{2r-1-3})
\end{eqnarray*}
and $f(x)\mid x^{2(2r-1)}-1$.
\item [(vi)] We have that 
$f(x)=x^{2r}+\left(\sum_{i=2}^{2r-2}(-1)^ix^{i}\right)+1=(1-x+x^2-\cdots +x^{2r-2})(x^2+x+1)$.
Suppose $r\equiv 0,1 \pmod 3$ and therefore $6(2r-1)\mid k$. Since $x^{6(2r-1)}-1$ divides $x^{k}-1$ it is enough to show that $f(x) \mid x^{6(2r-1)}-1$. We can write
\begin{eqnarray*}
x^{6(2r-1)}-1&=&(x^{3(2r-1)}-1)(x^{3(2r-1)}+1)\\
&=&(x^{3(2r-1)}-1)(x^{2r-1}+1)(x^{2(2r-1)}-x^{2r-1}+1)\\
&=&(x^3-1)(1+x^3+\cdots+x^{3(2r-1)-3})(x+1)\cdot\\
&&\qquad \cdot(1-x+x^2-\cdots +x^{2r-2})(x^{2(2r-1)}-x^{2r-1}+1)
\end{eqnarray*}
and $f(x)\mid x^{6(2r-1)}-1$. 
Suppose now $r\equiv 2 \pmod 3$ then $3\mid 2r-1$, $2(2r-1) \mid k$. Since $x^{2(2r-1)}-1$ divides $x^{k}-1$ it is enough to show that $f(x) \mid x^{2(2r-1)}-1$. We have that
\begin{eqnarray*}
x^{2(2r-1)}-1&=&(x^{2r-1}-1)(x^{2r-1}+1)\\
&=&(x^3-1)(1+x^3+\cdots+x^{2r-1-3})(x+1)(1-x+\cdots +x^{2r-2})
\end{eqnarray*}
and $f(x)\mid x^{2(2r-1)}-1$.
\end{enumerate}

\end{proof}

\end{document}